\documentclass[11pt]{article}
\usepackage{amsmath}
\usepackage{amssymb}
\usepackage{theorem}
\usepackage{pstricks}
\usepackage{euscript}
\usepackage{epic,eepic}
\PassOptionsToPackage{normalem}{ulem}
\definecolor{labelkey}{rgb}{0,0.08,0.45}
\definecolor{refkey}{rgb}{0,0.6,0.0}
\definecolor{Brown}{rgb}{0.45,0.0,0.05}
\definecolor{dgreen}{rgb}{0.00,0.49,0.00}
\definecolor{dblue}{rgb}{0,0.08,0.75}
\RequirePackage[dvips,colorlinks,hyperindex]{hyperref} 
\hypersetup{linktocpage=true,citecolor=dblue,linkcolor=dgreen}

\newcommand{\minimize}[2]{\ensuremath{\underset{\substack{{#1}}}%
{\mathrm{minimize}}\;\;#2 }}

\newcommand{\menge}[2]{\big\{{#1} \mid {#2}\big\}}

\newcommand{\emp}{\ensuremath{{\varnothing}}}

\newcommand{\scal}[2]{\left\langle{#1}\mid {#2} \right\rangle} 
\newcommand{\pscal}[2]{\boldsymbol{\langle}{#1}\mid{#2}\boldsymbol{\rangle}} 

\newcommand{\exi}{\ensuremath{\exists\,}}
\newcommand{\zeroun}{\ensuremath{\left]0,1\right[}}   
\newcommand{\HH}{\ensuremath{\mathcal H}}
\newcommand{\GG}{\ensuremath{\mathcal G}}

\newcommand{\sri}{\ensuremath{\operatorname{sri}}}
\newcommand{\RR}{\ensuremath{\mathbb R}}

\newcommand{\RPP}{\ensuremath{\,\left]0,+\infty\right[}}
\newcommand{\RX}{\ensuremath{\,\left]-\infty,+\infty\right]}}
\newcommand{\NN}{\ensuremath{\mathbb N}}
\newcommand{\dom}{\ensuremath{\operatorname{dom}}}

\newcommand{\gr}{\ensuremath{\operatorname{gra}}}
\newcommand{\prox}{\ensuremath{\operatorname{prox}}}

\newcommand{\cart}{\ensuremath{\mbox{\huge{$\times$}}}}

\newcommand{\Argmin}{\ensuremath{\operatorname{Argmin}}}

\newcommand{\zer}{\ensuremath{\operatorname{zer}}}

\newcommand{\Fix}{\ensuremath{\operatorname{Fix}}}
\newcommand{\Id}{\ensuremath{\operatorname{Id}}}

\newcommand{\weakly}{\ensuremath{\rightharpoonup}}

\newcommand{\pinf}{\ensuremath{+\infty}}
\newtheorem{theorem}{Theorem}[section]

\newtheorem{corollary}[theorem]{Corollary}
\newtheorem{proposition}[theorem]{Proposition}

\theoremstyle{plain}{\theorembodyfont{\rmfamily}
}
\theoremstyle{plain}{\theorembodyfont{\rmfamily}
}
\theoremstyle{plain}{\theorembodyfont{\rmfamily}
}
\theoremstyle{plain}{\theorembodyfont{\rmfamily}
}
\theoremstyle{plain}{\theorembodyfont{\rmfamily}
\newtheorem{problem}[theorem]{Problem}}
\theoremstyle{plain}{\theorembodyfont{\rmfamily}
\newtheorem{remark}[theorem]{Remark}}
\theoremstyle{plain}{\theorembodyfont{\rmfamily}
}

\numberwithin{equation}{section}

\begin{document}

\title{\sffamily\huge Forward-Douglas-Rachford splitting and
forward-partial inverse method for solving monotone inclusions 
\footnote{Contact author: 
L. M. Brice\~{n}o-Arias, {\ttfamily luis.briceno@usm.cl},
phone: +56 2 432 6662.
This work was supported by CONICYT under grant FONDECYT 3120054,
by ``Programa de financiamiento basal'' from the Center for
Mathematical Modeling, Universidad de Chile, by Anillo ACT 1106,
and by Project Math-Amsud N 13MATH01.}}

\author{Luis M. Brice\~{n}o-Arias$^1$\\[5mm]
\small $\!^1$Universidad T\'ecnica Federico Santa Mar\'ia\\
\small Departamento de Matem\'atica\\
\small Santiago, Chile\\[4mm]
}

\date{}
\maketitle

\vskip 8mm

\begin{abstract} \noindent
We provide two weakly convergent algorithms for finding a zero of the
sum of a maximally monotone operator, a cocoercive operator, and the
normal cone to a closed vector subspace of a real Hilbert space. The
methods exploit the intrinsic structure of the problem by activating
explicitly the cocoercive operator in the first step, and taking
advantage of a vector space decomposition in the second step. The
second step of the first method is a Douglas-Rachford iteration
involving the maximally monotone operator and the normal cone. In the
second method it is a proximal step involving the partial inverse of
the maximally monotone operator with respect to the vector subspace.
Connections between the proposed methods and other methods in the
literature are provided. Applications to monotone inclusions
with finitely many maximally monotone operators and
optimization problems are examined.
\end{abstract}

{\small 

}

\newpage

\section{Introduction}
This paper is concerned with the numerical resolution of the following problem.
\begin{problem}
\label{prob:1}
Let $\HH$ be a real Hilbert space and let $V$ be a closed vector 
subspace of $\HH$. The normal cone to $V$ is denoted by $N_V$.
Let $A\colon\HH\to 2^{\HH}$ be a maximally
monotone operator and let $B\colon \HH\to\HH$ be a
$\beta$--cocoercive 
operator in $V$, i.e., it
satisfies
\begin{equation}
\label{e:coco}
(\forall x\in V)(\forall y\in V)\quad\scal{x-y}{Bx-By}\geq\beta\|Bx-By\|^2.
\end{equation}
The problem is to
\begin{equation}
\label{e:mean0}
\text{find}\;\;x\in\HH\quad\text{such that}\quad 
0\in Ax+Bx+N_Vx,
\end{equation}
under the assumption that the set of solutions $Z$ of \eqref{e:mean0}
is nonempty. 
\end{problem}
 
Problem~\ref{prob:1} arises in a wide range of areas such
as optimization \cite{Invp08,Spin85},
variational inequalities \cite{Lion67,Tsen90,Tsen91}, 
monotone operator theory
\cite{Ecks92,Roc76a,Lion79,Spin83}, partial differential
equations \cite{Fort83,Glow89,Lion67,Merc80,Zeid90},
economics \cite{Jofr07,Penn12}, signal and image processing
\cite{Aujo05,Aujo06,Byrn02,Cham97,Daub04,Eick92,Rudi92,Sabh98,Vese03},
evolution inclusions \cite{Aubi90,Hara81,Show97},
and traffic theory
\cite{Beck56,Bert82,Fuku96,Rock84,Shef85}, among
others.

In the particular case when $B\equiv0$, \eqref{e:mean0} becomes
\begin{equation}
\label{e:pi}
\text{find}\;\;x\in V\quad\text{such that}\quad 
(\exi y\in V^{\bot})\quad y\in Ax,
\end{equation}
where $V^{\bot}$ stands for the orthogonal complement of $V$.
Problem \eqref{e:pi} has been studied in \cite{Spin83} and it
is solved 
with the method of partial inverses.
On the other hand, when $V=\HH$, \eqref{e:mean0} reduces  to
\begin{equation}
\label{e:fb}
\text{find}\;\;x\in\HH\quad\text{such that}\quad 
0\in Ax+Bx,
\end{equation}
which can be solved by the forward-backward splitting
(see \cite{Opti04} and the references therein).
In the general case, Problem~\ref{prob:1} can be solved by several
algorithms, but any of them exploits the intrinsic structure
of the problem. The forward-backward splitting \cite{Opti04} can solve
Problem~\ref{prob:1} by an explicit computation of $B$ and an
implicit computation involving the resolvent of $A+N_V$. The
disadvantage of this method is that this resolvent is not always
easy to compute. It is preferable, hence, to activate separately $A$
and $N_V$. In \cite{Pass79} an ergodic method involving the resolvents
of each maximally monotone operator separately is proposed, and weak
convergence to a solution to Problem~\ref{prob:1} is obtained.
However, the method includes vanishing parameters which leads to
numerical instabilities and, moreover, it involves the computation of
$(\Id+\gamma B)^{-1}$ for some positive constant $\gamma$, which is
not always easy to compute explicitly. The methods proposed in
\cite{Siopt3,Joca09,Comb12,Spin83} for finding a zero of the sum of
finitely many maximally
monotone operators overcomes the problem caused by the vanishing
parameters in \cite{Pass79}, but it still needs to compute
$(\Id+\gamma B)^{-1}$. The primal-dual method proposed in
\cite{Bang12}
overcomes the disadvantages of previous algorithms by computing
explicit steps of $B$. However, the method does not take advantage of
the vector subspace involved and, as a consequence, it needs to
store several auxiliary variables at each iteration, which can be
difficult for high dimensional problems.

In this paper we propose two methods for solving Problem~\ref{prob:1}
that exploit all the intrinsic properties of the problem. 
The first algorithm computes an explicit step on $B$ followed by a
Douglas-Rachford step \cite{Lion79,Svai10} involving $A$ and $N_V$.
The second method computes an explicit step on $B$ followed by an
implicit step involving the partial inverse of $A$ with respect to
$V$. The latter method generalizes the partial inverse method
\cite{Spin83} and the forward-backward
splitting \cite{Opti04} in the particular cases \eqref{e:pi} and
\eqref{e:fb}, respectively. We also provide connections between the
proposed methods, we study some relations with other methods in the
literature, and we illustrate the flexibility of this framework via
some applications. 

The paper is organized as follows. In Section~\ref{sec:1} we provide
the notation and some preliminaries. In Section~\ref{sec:2} we provide
a new version of the Krasnosel'ski\u{\i}-Mann iteration for the
composition of averaged operators. In
Section~\ref{sec:3} the latter method is applied to
particular averaged operators for obtaining the
forward-Douglas-Rachford splitting and in Section~\ref{sec:4} the
forward-partial inverse algorithm is proposed. We also
provide connections with other
algorithms in the literature. Finally, in Section~\ref{sec:5} we
examine an application for finding a zero of a sum of $m$ maximally
monotone operators and a cocoercive operator and an application to
optimization problems.

\section{Notation and preliminaries}
\label{sec:1}
Throughout this paper, $\HH$ is a real Hilbert space with 
scalar product denoted by $\scal{\cdot}{\cdot}$ and  
associated norm $\|\cdot\|$. The symbols $\weakly$ and $\to$
denote, respectively, weak and strong convergence and $\Id$ denotes
the identity operator. The indicator function of a subset $C$ of $\HH$
is
\begin{equation}
\label{e:iota}
\iota_C\colon x\mapsto
\begin{cases}
0,&\text{if}\;\;x\in C;\\
\pinf,&\text{if}\;\;x\notin C,
\end{cases}
\end{equation}
if $C$ is nonempty, closed, and convex, the projection of $x$ onto
$C$, denoted by $P_Cx$, is the unique point in $\Argmin_{y\in
C}\|x-y\|$, and the normal cone to $C$ is the maximally monotone
operator
\begin{equation}
\label{e:normalcone}
N_C\colon\HH\to 2^{\HH}\colon x\mapsto
\begin{cases}
\menge{u\in\HH}{(\forall y\in C)\:\:\scal{y-x}{u}\leq0},\quad
&\text{if }x\in C;\\
\emp,&\text{otherwise}.
\end{cases}
\end{equation}

An operator $R\colon\HH\to\HH$ is nonexpansive if
\begin{equation}
(\forall x\in\HH)(\forall y\in\HH)\quad\|Rx-Ry\|\leq\|x-y\|  
\end{equation}
and $\Fix R=\menge{x\in\HH}{Rx=x}$ is the set of fixed points of $R$.
An operator $T\colon\HH\to\HH$ is $\alpha$--averaged for some
$\alpha\in]0,1[$ if 
\begin{equation}
\label{e:averaged}
T=(1-\alpha)\Id+\alpha R 
\end{equation}
for some nonexpansive operator $R$, or, equivalently,
\begin{equation}
\label{e:averaged2}
(\forall x\in\HH)(\forall
y\in\HH)\quad\|Tx-Ty\|^2\leq\|x-y\|^2-\frac{1-\alpha}{\alpha}
\|(\Id-T)x-(\Id-T)y\|^2,
\end{equation}
or
\begin{equation}
\label{e:averaged3}
(\forall x\in\HH)(\forall
y\in\HH)\quad2(1-\alpha)\scal{x-y}{Tx-Ty}\geq\|Tx-Ty\|^2+(1-2\alpha)
\|x-y\|^2. 
\end{equation}
On the other hand, $T$ is $\beta$--cocoercive for
some $\beta\in\RPP$ if
\begin{equation}
\label{e:coco3}
(\forall x\in\HH)(\forall
y\in\HH)\quad\scal{x-y}{Tx-Ty}\geq\beta\|Tx-Ty\|^2.
\end{equation}
We say that $T$ is firmly nonexpansive if $T$ is $1/2$--averaged, or
equivalently, if $T$ is $1$--cocoercive.

We denote by $\gr A=\menge{(x,u)\in\HH\times\HH}{u\in Ax}$ the graph
of a set-valued operator $A\colon\HH\to 2^{\HH}$, by
$\dom A=\menge{x\in\HH}{Ax\neq\emp}$ its domain, by
$\zer A=\menge{x\in\HH}{0\in Ax}$ its set of zeros, and by
$J_A=(\Id+A)^{-1}$ its resolvent. If $A$ is monotone, then $J_A$ is 
single-valued and nonexpansive and, furthermore, if $A$ is maximally
monotone, then $\dom J_A=\HH$. Let $A\colon\HH\to 2^{\HH}$ be
maximally monotone. The reflection operator of $A$ is $R_A=2J_A-\Id$,
which is nonexpansive. The partial inverse of $A$ with respect to
a vector subspace $V$ of $\HH$, denoted by $A_V$, is defined by 
\begin{equation}
\label{e:partialinv}
(\forall (x,y)\in\HH^2)\quad y\in A_Vx\quad\Leftrightarrow\quad 
(P_Vy+P_{V^{\bot}}x)\in A(P_Vx+P_{V^{\bot}}y).
\end{equation}
For complements and further background 
on monotone operator theory, see 
\cite{Aubi90,Livre1,Spin83,Zali02,Zeid90}.

\section{Krasnosel'ski\u{\i}--Mann iterations for the composition of
averaged operators}
\label{sec:2}
The following result will be useful for obtaining the convergence of
the first method. It 
provides the weak convergence of the iterates generated by the
Krasnosel'ski\u{\i}--Mann iteration \cite{Opti04,Kras55,Mann53}
applied to the composition of finitely many averaged operators to a
common fixed
point. In \cite[Corollary~5.15]{Livre1} a similar method is proposed
with guaranteed convergence, but without including errors in the
computation of the operators involved. On the other hand, in
\cite{Opti04} another algorithm involving inexactitudes in the
computation of the averaged operators is studied in the case when
such operators may vary at each iteration. However, the relaxation
parameters in this case are forced to be in $\left]0,1\right]$. We
propose a new
method which includes summable errors in the computation of the
averaged operators and allows for a larger choice for the relaxation
parameters. First, for every strictly positive integer $i$ and 
a family of averaged operators $(T_j)_{1\leq j\leq m}$, let us define
\begin{equation}
\overset{m}{\underset{j=i}{\Pi}}T_j=
\begin{cases}
T_i\circ T_{i+1}\circ\cdots\circ T_m,\quad&\text{if }i\leq m;\\
\Id,&\text{otherwise}.
\end{cases}
\end{equation}

\begin{proposition}
\label{p:0}
Let $m\geq 1$, for every $i\in\{1,\ldots,m\}$, let
$\alpha_i\in\zeroun$, let $T_i$ be an $\alpha_i$-averaged
operator, and let $(e_{i,n})_{n\in\NN}$ be a sequence in 
$\HH$. In addition, set
\begin{equation}
\label{e:defalpha}
\alpha=\frac{m\,\max\{\alpha_{1},\ldots,\alpha_m\}}{1+(m-1)\max\{
\alpha_ { 1},\ldots,\alpha_m\}},
\end{equation}
let $(\lambda_n)_{n\in\NN}$ be a
sequence in 
$\left]0,1/\alpha\right[$,
suppose that $\Fix(T_{1}\circ\cdots\circ
T_{m})\neq\emp$, and suppose that
\begin{equation}
\label{e:condKM}
\sum_{n\in\NN}\lambda_n(1-\alpha\lambda_n)=\pinf\quad\text{and}\quad
(\forall
i\in\{1,\ldots,m\})\quad\sum_{n\in\NN}\lambda_n\|e_{i,n}\|<\pinf.
\end{equation}
Moreover, let $z_0\in\HH$ and set
\begin{equation}
\label{e:itrKM2}
(\forall n\in\NN)\quad
z_{n+1}=z_n+\lambda_n\Big(T_{1}\big(T_{2}(\cdots
T_{m-1}(T_{m}z_n+e_{m,n})+e_{m-1,n}\cdots)+e_{2,n}\big)+e_{1,n}
-z_n\Big). 
\end{equation}
Then the following hold for some
$\overline{z}\in\Fix(T_{1}\circ\cdots\circ
T_{m})$.
\begin{enumerate}
\item\label{p:0i} $z_n\weakly\overline{z}$.
\item\label{p:0ii}
$\displaystyle{\sum_{n\in\NN}
\lambda_n(1-\alpha\lambda_n)\Big\|\overset { m }
{ \underset { j=1 } { \Pi
} } T_j z_n-z_n\Big\|^2<\pinf}$.
\item\label{p:0iii-}$\overset{m}
{ \underset { j=1 } { \Pi
} } T_j z_n-z_n\to0$.
\item\label{p:0iii--}$z_{n+1}-z_n\to0$.
\item\label{p:0iii} 
$\displaystyle{\max_{1\leq i\leq m}\sum_{n\in\NN}\lambda_n
\Big\|(\Id-T_i)\overset{m}{\underset{j=i+1}{\Pi } }
T_jz_n-(\Id-T_i)\overset{m}{\underset{j=i+1}{\Pi}}
T_j\overline{z}\Big\|^2<\pinf}$.
\end{enumerate}
\end{proposition}
\begin{proof}
\ref{p:0i}:
First note that \eqref{e:itrKM2} can be written equivalently as 
\begin{equation}
\label{e:itrKM2aux1}
(\forall n\in\NN)\quad
z_{n+1}=z_n+\lambda_n(Tz_n+e_n-z_n),
\end{equation}
where 
\begin{equation}
\label{e:defTaux}
\begin{cases}
T=T_1\circ T_2\circ\cdots\circ T_m=\overset{m}{\underset{j=1}{\Pi}}T_j
z_n\\
e_n=T_{1}\big(T_{2}(\cdots
T_{m-1}(T_{m}z_n+e_{m,n})+e_{m-1,n}\cdots)+e_{2,n}\big)+e_{1,n}-
Tz_n.
\end{cases} 
\end{equation}
It follows from \cite[Lemma~2.2(iii)]{Opti04} that $T$ is
$\alpha$-averaged with $\alpha$ defined in \eqref{e:defalpha}, and,
using the nonexpansivity of $(T_i)_{1\leq
i\leq m}$, we obtain, for every $n\in\NN$,
\begin{align}
\|e_n\|&\leq\|T_{1}\big(T_{2}(\cdots
T_{m-1}(T_{m}z_n+e_{m,n})+e_{m-1,n}\cdots)+e_{2,n}\big)-T_{1}
\big(T_{2}(\cdots
T_{m-1}(T_{m}z_n)\cdots)\big)\|+\|e_{1,n}\|\nonumber\\
&\leq\|T_{2}\big(T_3(\cdots
T_{m-1}(T_{m}z_n+e_{m,n})+e_{m-1,n}\cdots)+e_{3,n}\big)-T_{2}
\big(T_3(\cdots
T_{m-1}(T_{m}z_n)\cdots)\big)\|\nonumber\\
&\hspace{13cm}+\|e_{2,n}\|+\|e_{1,n}\|\nonumber\\
&\leq \nonumber\\
&\hspace{.2cm}\vdots\nonumber\\
&\leq
\|T_{m-1}(T_{m}z_n+e_{m,n})-T_{m-1}(T_{m}z_n)\|+\|e_{m-1,n}
\|+\cdots+\|e_{2,n}\|+\|e_{1,n}\|\nonumber\\
&\leq \sum_{i=1}^m\|e_{i,n}\|.
\end{align}
Hence, it follows from \eqref{e:condKM} that
\begin{equation}
\label{e:auxconvf2}
\sum_{n\in\NN}\lambda_n\|e_n\|=\sum_{n\in\NN}
\lambda_n\sum_{i=1}^m\|e_{i,n}\|=\sum_{i=1}^m\sum_{n\in\NN}
\lambda_n\|e_{i,n}\|<\pinf.
\end{equation}
Now, set $R=(1-1/\alpha)\Id+(1/\alpha)T$ and, for every $n\in\NN$,
set $\mu_n=\alpha\lambda_n$. Then it follows from \eqref{e:averaged}
that $R$ is a nonexpansive operator, $\Fix R=\Fix T$, and
\eqref{e:itrKM2aux1} is equivalent to  
\begin{equation}
\label{e:itrKM}
(\forall n\in\NN)\quad z_{n+1}=z_n+\mu_n(Rz_n+c_n-z_n), 
\end{equation}
where $c_n=e_n/\alpha$. Since $(\mu_n)_{n\in\NN}$ is a sequence in
$\zeroun$ and \eqref{e:condKM} and \eqref{e:auxconvf2} yields
$\sum_{n\in\NN}\mu_n(1-\mu_n)=\pinf$ and
$\sum_{n\in\NN}\mu_n\|c_n\|<\pinf$, the result follows
from \cite[Lemma~5.1]{Opti04}. 

\ref{p:0ii}:
Fix $n\in\NN$. It
follows from \eqref{e:itrKM2aux1}, Cauchy-Schwartz inequality, and 
\cite[Lemma~2.13(ii)]{Livre1} that
\begin{align}
\label{e:auxconvf}
\|z_{n+1}-\overline{z}\|^2&=\|(1-\lambda_n)(z_n-\overline{z}
)+\lambda_n(Tz_n-T\overline{z}+e_n)\|^2\nonumber\\
&\leq
\|(1-\lambda_n)(z_n-\overline{z})
+\lambda_n(Tz_n-T\overline{z})\|^2+\varepsilon_n\nonumber\\
&=(1-\lambda_n)\|z_n-\overline{z}\|^2+\lambda_n\|Tz_n-T\overline{z}
\|^2-\lambda_n(1-\lambda_n)\|Tz_n-z_n\|^2+\varepsilon_n,
\end{align}
where,
\begin{equation}
(\forall k\in\NN)\quad
\varepsilon_k=\lambda_k^2\|e_k\|^2+2\lambda_k\|(1 -\lambda_k) 
(z_k -\overline{z})+\lambda_k(Tz_ k-T\overline{z})\|\|e_k\|. 
\end{equation}
Note that the convexity of $\|\cdot\|$, 
the nonexpansivity of $T$, and
\ref{p:0i} yield
\begin{align}
\sum_{k\in\NN}\varepsilon_k&=
\sum_{k\in\NN}\lambda_k^2\|e_k\|^2+2\sum_{k\in\NN}
\lambda_k\|(1-\lambda_k)(z_k-
\overline{z})+\lambda_k(Tz_
k-T\overline{z})\|\|e_k\|\nonumber\\
&\leq\Big(\sum_{k\in\NN}\lambda_k\|e_k\|\Big)^2+
2\sum_{k\in\NN}\lambda_k\big((1-\lambda_k)\|z_k-
\overline{z}\|+\lambda_k\|Tz_k-T\overline{z}\|\big)\|e_k\|\nonumber\\
&\leq\Big(\sum_{k\in\NN}\lambda_k\|e_k\|\Big)^2+
2\Big(\sup_{k\in\NN}\|z_k-\overline{z}\|\Big)\sum_{k\in\NN}
\lambda_k\|e_k\|<\pinf.
\end{align}
On one hand, since $T$ is $\alpha$-averaged, it follows from
\eqref{e:auxconvf} and \eqref{e:averaged2} that
\begin{align}
\|z_{n+1}-\overline{z}\|^2
&\leq
(1-\lambda_n)\|z_n-\overline{z}\|^2+\lambda_n\Big(\|z_n-\overline{z}
\|^2-\frac{(1-\alpha)}{\alpha}
\|Tz_n-z_n\|^2\Big)\nonumber\\
&\hspace{8cm}
-\lambda_n(1-\lambda_n)\|Tz_n-z_n\|^2+\varepsilon_n\nonumber\\
&\leq
\|z_{n}-\overline{z}\|^2-\frac{\lambda_n(1-\alpha\lambda_n)}{
\alpha }\|Tz_n-z_n\|^2+\varepsilon_n,
\end{align}
and, hence, the result is deduced from \cite[Lemma~3.1(iii)]{Comb01}.

\ref{p:0iii-}:
It follows from \eqref{e:condKM} and \ref{p:0ii} that 
$\varliminf\|Tz_n-z_n\|=0$. Moreover, it follows from
\eqref{e:itrKM2aux1} that
\begin{align}
(\forall n\in\NN)\quad
\|Tz_{n+1}-z_{n+1}\|&\leq\|Tz_{n+1}
-Tz_n\|+(1-\lambda_n)\|Tz_n-z_n\|+\lambda_n\|e_n\|\nonumber\\
&\leq\|z_{n+1}-z_n\|+
(1-\lambda_n)\|Tz_n-z_n\|+\lambda_n\|e_n\|\nonumber\\
&\leq\|Tz_n-z_n\|+2\lambda_n\|e_n\|.
\end{align}
Hence, from \eqref{e:auxconvf2} and \cite[Lemma~3.1]{Comb01} 
we deduce that 
$(\|Tz_n-z_n\|)_{n\in\NN}$ converges, and therefore, 
$Tz_n-z_n\to0$.

\ref{p:0iii--}: From \eqref{e:itrKM2aux1}, \eqref{e:defTaux},
\ref{p:0iii-}, and \eqref{e:auxconvf2} we obtain
\begin{equation}
\|z_{n+1}-z_n\|\leq\lambda_n\|Tz_n-z_n\|+\lambda_n\|e_n\|\leq
(1/\alpha)\|Tz_n-z_n\|+\lambda_n\|e_n\|\to0.
\end{equation}

\ref{p:0iii}: Since $(T_i)_{1\leq
i\leq m}$ are averaged operators, we have from \eqref{e:defTaux}
and \eqref{e:averaged2} that
\begin{align}
\|Tz_n-T\overline{z}\|^2&\leq\Big\|\overset{m}{\underset{j=2}{\Pi}}
T_jz_n-\overset{m}{\underset{j=2}{\Pi}}T_j\overline{z}\Big\|^2
-\frac{1-\alpha_1}{\alpha_1}\Big\|(\Id-T_1)\overset{m}{\underset{j=2}{
\Pi}}T_jz_n-(\Id-T_1)\overset{m}{\underset{j=2}{\Pi}}T_j\overline{z}
\Big\|^2\nonumber\\
&\leq\Big\|\overset{m}{\underset{j=3}{\Pi}}T_jz_n-\overset{m}{
\underset{j=3}{\Pi}}T_j\overline{z}\Big\|^2-\frac{1-\alpha_2}{\alpha_2
}\Big\|(\Id-T_2)\overset{m}{\underset{j=3}{ \Pi}}
T_jz_n-(\Id-T_2)\overset{m}{\underset{j=3}{\Pi}}T_j\overline{z}
\Big\|^2\nonumber\\
&\hspace{3.981cm}-\frac{1-\alpha_1}{\alpha_1}\Big\|(\Id-T_1)\overset{m
}{\underset{j=2}{\Pi}}T_jz_n-(\Id-T_1)\overset{m}{\underset{j=2}{\Pi}}
T_j\overline{z}
\Big\|^2\nonumber\\
&\hspace{.2cm}\vdots\nonumber\\
&\leq
\|z_n-\overline{z}\|^2-\sum_{i=1}^m\frac{1-\alpha_i}{\alpha_i}
\Big\|(\Id-T_i)
\overset{m}{\underset{j=i+1}{\Pi}}
T_jz_n-(\Id-T_i)\overset{m}{\underset{j=i+1}{\Pi}}T_j\overline{z}
\Big\|^2.
\end{align}
Hence, from \eqref{e:auxconvf} we deduce
\begin{equation}
\|z_{n+1}-\overline{z}\|^2\leq
\|z_n-\overline{z}\|^2-\lambda_n\sum_{i=1}^m\frac{1-\alpha_i}{\alpha_i
}\Big\|(\Id-T_i)\overset{m
}{\underset{j=i+1}{\Pi } }
T_jz_n-(\Id-T_i)\overset{m}{\underset{j=i+1}{\Pi}}
T_j\overline{z}\Big\|^2+\varepsilon_n.
\end{equation}
Therefore, it follows from \cite[Lemma~3.1(iii)]{Comb01} that 
\begin{multline}
\sum_{i=1}^m\frac{1-\alpha_i}{\alpha_i}\sum_{n\in\NN}
\lambda_n\Big\|(\Id-T_i)\overset{m
}{\underset{j=i+1}{\Pi } }
T_jz_n-(\Id-T_i)\overset{m}{\underset{j=i+1}{\Pi}}
T_j\overline{z}\Big\|^2\\=\sum_{n\in\NN}\lambda_n\sum_{i=1}^m\frac{
1-\alpha_i}{
\alpha_i}\Big\|(\Id-T_i)\overset{m
}{\underset{j=i+1}{\Pi } }
T_jz_n-(\Id-T_i)\overset{m}{\underset{j=i+1}{\Pi}}
T_j\overline{z}\Big\|^2<\pinf,
\end{multline}
which yields the result.
\end{proof}

\begin{remark}
In the particular case when $m=1$, Proposition~\ref{p:0} provides the
weak convergence of the iterates generated
by the classical Krasnosel'ski\u{\i}-Mann iteration
\cite{Opti04,Kras55,Mann53}
in the case of averaged operators. This result is interesting in this
own right since it generalizes \cite[Proposition~5.15]{Livre1} by
considering errors on the computation of the involved operator and
provides a larger choice of relaxation parameters than in the
nonexpansive case (see, e.g.,\cite{Opti04,Kras55,Mann53}).
\end{remark}

\section{Forward-Douglas-Rachford splitting}
\label{sec:3}
In this section we provide the first method for solving
Problem~\ref{prob:1}. We provide a characterization of the
solutions to Problem~\ref{prob:1}, then the algorithm is proposed and
its weak convergence to a solution to Problem~\ref{prob:1} is proved.

\subsection{Characterization}
Let us start with a characterization of the solutions to
Problem~\ref{prob:1}.

\begin{proposition}
\label{p:1}
Let $\gamma\in\left]0,2\beta\right[$ and $\HH$, $V$, $A$, $B$, and $Z$
be as in Problem~\ref{prob:1}.
Define 
\begin{equation}
\label{e:defmaxmon}
\begin{cases}
T_{\gamma}=\frac12(\Id+R_{\gamma A}\circ R_{N_V})\colon\HH\to
\HH\\
S_{\gamma}=\Id-\gamma P_V\circ B\circ P_V\colon\HH\to\HH.
\end{cases} 
\end{equation}
Then the following hold.
\begin{enumerate}
\item\label{p:1i} $T_{\gamma}$ is firmly nonexpansive.
\item\label{p:1ii} $S_{\gamma}$ is $\gamma/(2\beta)$--averaged.
\item\label{p:1iii} Let $x\in\HH$. Then $x\in Z$ 
if and only if 
\begin{equation}
x\in V\quad\text{and}\quad
(\exi y\in V^{\bot}\cap(Ax+Bx))\quad\text{such that}\quad 
x-\gamma (y-P_{V^{\bot}}Bx)\in\Fix(T_{\gamma}\circ S_{\gamma}). 
\end{equation}
\end{enumerate}
\end{proposition}
\begin{proof}
\ref{p:1i}: Since $\gamma A$ is maximally monotone $J_{\gamma A}$
is firmly nonexpansive and $R_{\gamma A}=2J_{\gamma A}-\Id$ is
nonexpansive. An analogous argument yields the
nonexpansivity of $R_{N_V}=2P_V-\Id$. Hence, $R_{\gamma A}\circ
R_{N_V}$ is nonexpansive and the result 
follows from \eqref{e:averaged}.

\ref{p:1ii}: Since $V$ is a closed vector subspace of $\HH$ we have
that $P_V$ is linear and $P_V^*=P_V$. Hence, the cocoercivity of $B$
in $V$ yields, for every $(z,w)\in\HH^2$ and
$\gamma\in\left]0,2\beta\right[$,
\begin{align}
\scal{z-w}{\gamma P_V\big(B(P_Vz)\big)-
\gamma P_V\big(B(P_Vw)\big)}&=\gamma\scal{P_Vz-P_Vw}{B(P_Vz)-
B(P_Vw)}\nonumber\\
&\geq\gamma\beta\|B(P_Vz)- B(P_Vw)\|^2\nonumber\\
&\geq(\beta/\gamma)\|\gamma P_V\big(B(P_Vz)\big)-
\gamma P_V\big(B(P_Vw)\big)\|^2.
\end{align}
Since $\gamma\in\left]0,2\beta\right[$ the result follows
from \cite[Lemma~2.3]{Opti04}.

\ref{p:1iii}: Let $x\in\HH$ be a solution to Problem~\ref{prob:1}. 
We have $x\in V$ and there exists $y\in V^{\bot}=N_Vx$ such that 
$y\in Ax+Bx$. Set
$z=x-\gamma (y-P_{V^{\bot}}Bx)$. Note that
$R_{N_V}z=2P_Vz-z=x+\gamma (y-P_{V^{\bot}}Bx)$ and $P_Vz=x$. Hence,
since $B$ is single valued and, for every $w\in V$, $R_Vw=w$, it
follows from the linearity of $P_V$
that
\begin{equation}
\label{e:carac0-}
x+\gamma y-\gamma Bx=x+\gamma(y-P_{V^{\bot}}Bx)-\gamma
P_VBx=R_{N_V}z-\gamma
P_VBP_Vz=R_{N_V}(z-\gamma
P_VBP_Vz),
\end{equation}
and, therefore, 
\begin{align}
\label{e:carac0}
y\in Ax+Bx\quad&\Leftrightarrow\quad
x+\gamma y-\gamma Bx\in x+\gamma Ax\nonumber\\
&\Leftrightarrow\quad
x=J_{\gamma A}(x+\gamma y-\gamma Bx)\nonumber\\
&\Leftrightarrow\quad
x=J_{\gamma A}\big(R_{N_V}(z-\gamma
P_VBP_Vz)\big)\nonumber\\
&\Leftrightarrow\quad
x=\frac12\Big(2J_{\gamma A}\big(R_{N_V}(z-\gamma
P_VBP_Vz)\big)-R_{N_V}(z-\gamma
P_VBP_Vz)+x+\gamma y-\gamma Bx\Big)\nonumber\\
&\Leftrightarrow\quad
x=\frac12\Big(R_{\gamma A}\big(R_{N_V}(z-\gamma
P_VBP_Vz)\big)+x+\gamma y-\gamma Bx\Big)\nonumber\\
&\Leftrightarrow\quad
x=\frac12\Big(R_{\gamma A}\big(R_{N_V}(z-\gamma
P_VBP_Vz)\big)+z-\gamma
P_VBP_Vz\Big)+\gamma (y-P_{V^{\bot}}Bx)\nonumber\\
&\Leftrightarrow\quad
z=T_{\gamma}\circ S_{\gamma}z,
\end{align}
which yields the result.
\end{proof}

\subsection{Algorithm and convergence}
In the following result we propose the algorithm and we prove its
convergence to a solution to
Problem~\ref{prob:1}. The method is inspired from the characterization
provided in Proposition~\ref{p:1} and Proposition~\ref{p:0}.
\begin{theorem}
\label{t:0}
Let $\HH$, $V$, $A$, $B$, and $Z$ be as in Problem~\ref{prob:1},
let $\gamma\in\left]0,2\beta\right[$, let
$\alpha=\max\{2/3,2\gamma/(\gamma+2\beta)\}$,
let $(\lambda_n)_{n\in\NN}$ be a sequence in
$\left]0,1/\alpha\right[$,
let $(a_n)_{n\in\NN}$ and $(b_n)_{n\in\NN}$ be sequences in $\HH$,
and suppose that
\begin{equation}
\label{e:conderro}
\sum_{n\in\NN}\lambda_n(1-\alpha\lambda_n)=\pinf\quad\text{and}\quad
\sum_{n\in\NN}\lambda_n(\|a_n\|+\|b_n\|)<\pinf.
\end{equation}
Moreover, let $z_0\in\HH$ and set
\begin{align}
\label{e:algo0}
(\forall n\in\NN)\quad
&\left 
\lfloor 
\begin{array}{l}
x_n=P_Vz_n\\
y_n=(x_n-z_n)/\gamma\\
s_n=x_n-\gamma P_V\big(Bx_n+a_n\big)+\gamma y_n\\
p_n=J_{\gamma A}s_n+b_n\\
z_{n+1}=z_n+\lambda_n(p_n-x_n).
\end{array}
\right.
\end{align}
Then the sequences $(x_n)_{n\in\NN}$ and $(y_n)_{n\in\NN}$ are in 
$V$ and $V^{\bot}$, respectively, and the following hold for 
some $\overline{x}\in Z$ and some $\overline{y}\in
V^{\bot}\cap\big(A\overline{x}+P_VB\overline{x}\big)$.
\begin{enumerate}
\item\label{t:0i} $x_n\weakly \overline{x}\,$  and  $\,y_n\weakly
\overline{y}$.
\item\label{t:0ii} $x_{n+1}-x_n\to 0\,$  and  $\,y_{n+1}-y_n\to 0$.
\item\label{t:0iii}
$\sum_{n\in\NN}\lambda_n\|P_V(Bx_n-B\overline{x})\|^2<\pinf$.
\end{enumerate}
\end{theorem}
\begin{proof}
First note that \eqref{e:algo0} can be written equivalently as
\begin{align}
\label{e:aux1}
(\forall n\in\NN)\quad 
&\left 
\lfloor 
\begin{array}{l}
x_n=P_Vz_n\\
y_n=-P_{V^{\bot}}z_n/\gamma\\
z_{n+1}=z_n+\lambda_n\big(T_{\gamma}(S_{\gamma}z_n+c_{n})+b_n
-z_n\big),
\end{array}
\right.
\end{align}
where $T_{\gamma}$ and $S_{\gamma}$ are defined in
\eqref{e:defmaxmon} and, for every $n\in\NN$,  
$c_{n}=-\gamma P_Va_n$. We have from \eqref{e:conderro} that
\begin{equation}
\sum_{n\in\NN}\lambda_n(\|b_n\|+\|c_n\|)\leq 
\sum_{n\in\NN}\lambda_n(\|b_n\|+\gamma\|a_n\|)\leq\max\{1,\gamma\}
\sum_{n\in\NN}\lambda_n(\|a_n\|+\|b_n\|)<\pinf.
\end{equation}
Moreover, it follows from
Proposition~\ref{p:1}\ref{p:1i}\&\ref{p:1ii} that $T_{\gamma}$
is $1/2$-averaged and $S_{\gamma}$ is $\gamma/(2\beta)$-averaged.
Altogether, by setting $m=2$, $T_1=T_{\gamma}$,
$T_2=S_{\gamma}$, $\alpha_1=1/2$, $\alpha_2=\gamma/(2\beta)$,
$e_{1,n}=b_n$, $e_{2,n}=c_n$, and noting that 
\begin{equation}
\frac{2\max\{1/2,\gamma/(2\beta)\}}{1+\max\{1/2,\gamma/(2\beta)\}}
=\max\{2/3,2\gamma/(\gamma+2\beta)\}=\alpha, 
\end{equation}
it follows from Proposition~\ref{p:0} that there exists
$\overline{z}\in\Fix(T_{\gamma}\circ S_{\gamma})$ such that
\begin{eqnarray}
\label{e:conc1}
& &z_n\weakly\overline{z}\\
\label{e:conc2}
& &z_{n+1}-z_n\to0 \\
\label{e:conc3}
& & \sum_{n\in\NN}\lambda_n\|(\Id-S_{\gamma})z_n-(\Id-S_{\gamma}
)\overline {z}\|^2<\pinf.
\end{eqnarray}
Now set
$\overline{x}=P_V\overline{z}$ and
$\overline{y}=-P_{V^{\bot}}\overline{z}/\gamma$. It follows from
Proposition~\ref{p:1}\ref{p:1iii} that $\overline{x}$ is solution
to Problem~\ref{prob:1} and $\overline{y}=y-P_{V^{\bot}}B\overline{x}$
for some $y\in V^{\bot}\cap(A\overline{x}+B\overline{x})$. Then,
$\overline{y}\in V^{\bot}\cap(A\overline{x}+P_VB\overline{x})$.

\ref{t:0i}: It is clear from \eqref{e:aux1} and
\eqref{e:conc1} that 
$x_n\weakly\overline{x}$ and 
$y_n\weakly\overline{y}$.

\ref{t:0ii}: It is a consequence of \eqref{e:conc2} and
\begin{equation}
(\forall n\in\NN)\quad
\|z_{n+1}-z_{n}\|^2=\|x_{n+1}-x_{n}\|^2+\gamma^2\|y_{n+1}-y_n\|^2. 
\end{equation}

\ref{t:0iii}: It follows from \eqref{e:defmaxmon} that
\begin{equation}
(\forall n\in\NN)\quad\|(\Id-S_{\gamma})z_n-(\Id-S_{\gamma}
)\overline {z}\|^2=\gamma\|P_V(Bx_n)-P_V(B\overline{x})\|^2.
\end{equation}
Hence, the result follows from \eqref{e:conc3}.
\end{proof}
\begin{remark}
Note that, if $\varliminf\lambda_n>0$, then
Theorem~\ref{t:0}\ref{t:0iii} implies
$P_V(Bx_n)\to P_V(B\overline{x})$. 
 \end{remark}

\section{Forward-partial-inverse splitting}
\label{sec:4}
We provide a second characterization of solutions to
Problem~\ref{prob:1} via the partial inverse operator introduced in
\cite{Spin83}. This characterization motivates a second algorithm,
whose convergence to a solution to Problem~\ref{prob:1} is proved.
The proposed method generalizes the partial inverse method proposed
in \cite{Spin83} and the forward-backward splitting \cite{Opti04}.
\subsection{Characterization}
\begin{proposition}
\label{p:12}
Let $\gamma\in\RPP$ and $\HH$, $A$, $B$, and $V$ be as in
Problem~\ref{prob:1}.
Define 
\begin{equation}
\label{e:defmaxmon2}
\begin{cases}
\mathcal{A}_{\gamma}=(\gamma A)_V\colon\HH\to 2^{\HH}\\
\mathcal{B}_{\gamma}=\gamma P_V\circ B\circ P_V\colon\HH\to V.
\end{cases} 
\end{equation}
Then the following hold.
\begin{enumerate}
\item\label{p:12i} $\mathcal{A}_{\gamma}$ is maximally monotone.
\item\label{p:12ii} $\mathcal{B}_{\gamma}$ is
$\beta/\gamma$--cocoercive.
\item\label{p:12iii} Let $x\in\HH$. Then $x$ is a solution to
Problem~\ref{prob:1}
if and only if 
\begin{equation}
x\in V\quad\text{and}\quad
(\exi y\in V^{\bot}\cap(Ax+Bx))\quad\text{such that}\quad 
x+\gamma
(y-P_{V^{\bot}}Bx)\in\zer(\mathcal{A}_{\gamma}+\mathcal{B}_{\gamma}). 
\end{equation}

\end{enumerate}

\end{proposition}
\begin{proof}
\ref{p:1i}: Since $\gamma A$ is maximally monotone, the result 
follows from \cite[Proposition~2.1]{Spin83}.
\ref{p:1ii}: It follows from \eqref{e:defmaxmon} that
$\mathcal{B}_{\gamma}=\Id-S_{\gamma}$ and from
Proposition~\ref{p:1}\ref{p:1ii} and \eqref{e:averaged} we deduce
that there exists a nonexpansive operator $R_{\gamma}\colon\HH\to\HH$
such that
$\mathcal{B}_{\gamma}=\Id-(1-\gamma/(2\beta))\Id-\gamma/(2\beta)R_{
\gamma}=(\gamma/\beta)(\Id-R_{\gamma})/2$. Hence, since
$(\Id-R_{\gamma})/2$ is firmly nonexpansive, the result follows from 
\eqref{e:coco3}.
\ref{p:1iii}:
Let $x\in\HH$ be a solution to Problem~\ref{prob:1}. 
We have $x\in V$ and there exists $y\in V^{\bot}=N_Vx$ such that 
$y\in Ax+Bx$. 
Since $B$ is single valued and $P_V$ is linear, it follows from
\eqref{e:partialinv} that
\begin{align}
\label{e:carac1}
y\in Ax+Bx\quad
&\Leftrightarrow\quad
\gamma y-\gamma Bx\in \gamma Ax\nonumber\\
&\Leftrightarrow\quad
-\gamma P_V(Bx)
\in (\gamma A)_V\big(x+\gamma (y-P_{V^{\bot}}Bx)\big)\nonumber\\
&\Leftrightarrow\quad
0\in (\gamma A)_V(x+\gamma (y-P_{V^{\bot}}Bx))+\gamma
P_V\big(B\big(P_V(x+\gamma (y-P_{V^{\bot}}Bx))\big)\big)\nonumber\\
&\Leftrightarrow\quad
x+\gamma (y-P_{V^{\bot}}Bx)\in\zer
(\mathcal{A}_{\gamma}+\mathcal{B}_{\gamma}),
\end{align}
which yields the result.
\end{proof}
\begin{remark}
Note that the characterizations provided in Proposition~\ref{p:1} y
Proposition~\ref{p:12} are related. Indeed,
Proposition~\ref{p:1}\ref{p:1iii} and
Proposition~\ref{p:12}\ref{p:1iii} yield
\begin{equation}
Z=P_V(\Fix(T_{\gamma}\circ
S_{\gamma}))=P_V(\zer(\mathcal{A}_{\gamma}+\mathcal{B}_{\gamma}))
\quad\text{and}\quad R_{N_V}(\Fix(T_{\gamma}\circ
S_{\gamma}))=\zer(\mathcal{A}_{\gamma}+\mathcal{B}_{\gamma}).
\end{equation}
\end{remark}

\subsection{Algorithm and convergence}
\begin{theorem}
\label{t:1}
Let $\HH$, $V$, $A$, $B$, and $Z$ be as in Problem~\ref{prob:1},
let $\gamma\in\RPP$, let
$\varepsilon\in\left]0,\max\{1,\beta/\gamma\}\right[$, 
let $(\delta_n)_{n\in\NN}$ be a sequence in
$\left[\varepsilon,(2\beta/\gamma)-\varepsilon\right]$, and let
$(\lambda_n)_{n\in\NN}$ be a sequence in $\left[\varepsilon,1\right]$.
Moreover, 
let $x_0\in V$, let $y_0\in V^{\bot}$, and, for every $n\in\NN$, 
consider the following routine.
\begin{align}
\text{Step}\:1.\:&
\text{\rm Find }(p_n,q_n)\in\HH^2\:\text{\rm such that }
x_n-\delta_n\gamma P_VBx_n+\gamma y_n=p_n+\gamma
q_n\nonumber\\
\label{e:auxprop}
&\text{\rm and } \frac{P_Vq_n}{\delta_n}+P_{V^{\bot}}q_n\in
A\Big(P_Vp_n+\frac{P_{V^{\bot}}p_n}{\delta_n}\Big).\\
\text{Step}\:2.\:&\text{\rm Set }
x_{n+1}=x_n+\lambda_n(P_Vp_n-x_n)\:\:\text{\rm and }
y_{n+1}=y_n+\lambda_n(P_{V^{\bot}}q_n-y_n).\: \text{\rm Go to
}\text{Step}\:1.\nonumber
\end{align}
Then, the sequences $(x_n)_{n\in\NN}$ and $(y_n)_{n\in\NN}$ are in 
$V$ and $V^{\bot}$, respectively, and the following hold for 
some $\overline{x}\in Z$ and $\overline{y}\in
V^{\bot}\cap(A\overline{x}+P_VB\overline{x})$.
\begin{enumerate}
\item\label{t:1i} $x_n\weakly \overline{x}\,$ and $\,y_n\weakly
\overline{y}$.
\item\label{t:1ii} $x_{n+1}-x_n\to 0\,$ and $\,y_{n+1}-y_n\to 0$.
\item\label{t:1iii} $P_VBx_n\to P_VB\overline{x}$.
\end{enumerate}
\end{theorem}
\begin{proof}
Since $x_0\in V$ and $y_0\in V^{\bot}$, \eqref{e:auxprop}
yields $(x_n)_{n\in\NN}\subset V$ and $(y_n)_{n\in\NN}\subset V^{\bot}$.
Thus, for every $n\in\NN$, it follows from \eqref{e:auxprop} and the
linearity of $P_V$ and $P_{V^{\bot}}$ that
\begin{equation}
\label{e:auxpof}
\begin{cases}
\hskip.2cmP_Vp_n+\gamma P_Vq_n&=P_V(p_n+\gamma
q_n)=P_V(x_n-\delta_n\gamma P_VBx_n+\gamma y_n)
=x_n-\delta_n\gamma P_VBx_n\\
P_{V^{\bot}}p_n+\gamma
P_{V^{\bot}}q_n\hskip-.5cm&=P_{V^{\bot}}(p_n+\gamma q_n)
=P_{V^{\bot}}(x_n-\delta_n\gamma P_VBx_n+\gamma y_n)
=\gamma y_n,
\end{cases}
\end{equation}
which yield
\begin{equation}
\label{e:aux222}
\begin{cases}
\hskip.05cmP_Vq_n&=(x_n-\delta_n\gamma
Bx_n-P_Vp_n)/\gamma=(x_n-x_{n+1}
)/(\gamma\lambda_n)-\delta_nP_VBx_n\\
P_{V^{\bot}}p_n\hskip-1cm&=\gamma(y_n-P_{V^{\bot}}q_n)=\gamma(y_n-y_{
n+1})/\lambda_n.
\end{cases}
\end{equation}
On the other hand, from \eqref{e:auxprop} we obtain 
\begin{equation}
P_Vp_n=x_n+\frac{x_{n+1}-x_n}{\lambda_n}\quad\text{and}\quad
P_{V^{\bot}}q_n=y_n+\frac{y_{n+1}-y_n}{\lambda_n}.
\end{equation}
Hence, it follows from \eqref{e:aux222} and \eqref{e:auxprop} that
\begin{equation}
\frac{(x_n-x_{n+1})}{\lambda_n\delta_n\gamma}-
P_VBx_n+y_{n}+\frac{y_{n+1}-y_n
}{\lambda_n} \in
A\Big(x_n+\frac{x_{n+1}-x_n}{\lambda_n}
+\frac{\gamma(y_n-y_{n+1})}{\lambda_n\delta_n}\Big), 
\end{equation}
or equivalently,
\begin{equation}
\frac{(x_n-x_{n+1})}{\lambda_n\delta_n}-
\gamma P_VBx_n+\gamma y_{n}+\frac{\gamma(y_{n+1}-y_n)}{\lambda_n}\in
\gamma A\Big(x_n+\frac{x_{n+1}-x_n}{\lambda_n}
+\frac{\gamma(y_n-y_{n+1})}{\lambda_n\delta_n}\Big).
\end{equation}
Thus, by using the definition of partial inverse \eqref{e:partialinv}
we obtain
\begin{equation}
\frac{(x_n-x_{n+1})}{\lambda_n\delta_n}-\gamma
P_VBx_n+\frac{\gamma(y_n-y_{n+1})}{\lambda_n\delta_n}
\in(\gamma
A)_V\bigg(x_n+\gamma
y_{n}+\frac{x_{n+1}-x_n+\gamma(y_{n+1}-y_n)}{\lambda_n}\bigg),
\end{equation}
which can be written equivalently as
\begin{multline}
x_n+\gamma y_n-\delta_n\gamma P_VBx_n-\bigg(x_n+\gamma
y_n+\frac{x_{n+1}-x_n+\gamma(y_{n+1}-y_n)}{\lambda_n}\bigg)\\
\in\delta_n(\gamma A)_V\bigg(x_n+\gamma
y_{n}+\frac{x_{n+1}-x_n+\gamma(y_{n+1}-y_n)}{\lambda_n}\bigg).
\end{multline}
Hence, we have
\begin{equation}
x_n+\gamma
y_{n}+\frac{x_{n+1}-x_n+\gamma(y_{n+1}-y_n)}{\lambda_n}=J_{
\delta_n(\gamma A)_V}(x_n+\gamma
y_n-\delta_n\gamma P_VBx_n),
\end{equation}
or equivalently,
\begin{equation}
\label{e:aux333}
x_{n+1}+\gamma y_{n+1}=x_n+\gamma y_n+\lambda_n\Big(J_{
\delta_n(\gamma A)_V}(x_n+\gamma
y_n-\delta_n\gamma P_VBx_n)-x_n+\gamma y_n\Big).
\end{equation}
If, for every $n\in\NN$, we denote $r_n=x_n+\gamma y_n$, from 
\eqref{e:aux333} and \eqref{e:defmaxmon2} we have
\begin{equation}
\label{e:FB2}
r_{n+1}=r_n+\lambda_n\big(J_{\delta_n\mathcal{A}_{\gamma}}
(r_n-\delta_n\mathcal{B}_{\gamma}r_n)-r_n\big).
\end{equation}
Since $(\delta_n)_{n\in\NN}\subset
\left[\varepsilon,2(\beta/\gamma)-\varepsilon\right]$, it follows from
Proposition~\ref{p:12}\ref{p:12i}\&\ref{p:12ii} and
\cite[Theorem~2.8]{Sico10} 
that there exists
$\overline{r}\in\zer(\mathcal{A}_{\gamma}+\mathcal{B}_{\gamma})$ such
that 
$r_n\weakly \overline{r}$, $\mathcal{B}_{\gamma}r_n\to
\mathcal{B}_{\gamma}\overline{r}$,
$r_n-r_{n+1}=\lambda_n(r_n-J_{\delta_n\mathcal{A}_{\gamma}}
(r_n-\delta_n\mathcal{B}_{\gamma}r_n))\to 0$.
Hence, by taking $\overline{x}=P_V\overline{r}$ and
$\overline{y}=P_{V^{\bot}}\overline{r}/\gamma$, 
Proposition~\ref{p:12}\ref{p:12iii} asserts that
$\overline{x}\in Z$, $\overline{y}\in
V^{\bot}\cap(A\overline{x}+P_VB\overline{x})$, and the results follow
from 
\begin{equation}
(\forall (x,y)\in\HH^2)\quad\scal{x}{y}=\scal{P_Vx}{P_Vy}+
\scal{P_{V^{\bot}}x}{P_{V^{\bot}}y}
\end{equation}
and the definition of $\mathcal{B}_{\gamma}$.
\end{proof}

\begin{remark}\
\begin{enumerate}
\item 
It is known that the forward--backward
splitting admits errors in the computations of the
operators involved \cite{Opti04}. In our algorithm these inexactitudes
have not been considered for simplicity.
\item 
In the particular case when $\gamma<2\beta$,
$\lambda_n\equiv1$, and $B\equiv0$, the
forward-partial-inverse method reduces to the partial inverse method
proposed in \cite{Spin83} for solving \eqref{e:pi}.
\end{enumerate}
\end{remark}

The sequence $(\delta_n)_{n\in\NN}$ in Theorem~\ref{t:1} can be
manipulated 
in order to accelerate the algorithm. However, as in \cite{Spin83}, 
{\em Step} 1 in Theorem~\ref{t:1} is not always easy to compute. 
The following result show us a particular case of our 
method in which {\em Step} 1 can be obtained explicitly 
when the resolvent of $A$ is computable.

\begin{corollary}
\label{c:0}
Let $\gamma\in\left]0,2\beta\right[$, let $x_0\in V$, let $y_0\in
V^{\bot}$, let $(\lambda_n)_{n\in\NN}$ be a sequence in
$\left[\varepsilon,1\right]$, and consider the following routine.
\begin{align}
\label{e:algo}
(\forall n\in\NN)\quad
&\left 
\lfloor 
\begin{array}{l}
s_n=x_n-\gamma P_V Bx_n+\gamma y_n\\
p_n=J_{\gamma A}s_n\\
y_{n+1}=y_n+(\lambda_n/\gamma)(P_Vp_n-p_n)\\
x_{n+1}=x_n+\lambda_n(P_Vp_n-x_n).
\end{array}
\right.
\end{align}
Then, the sequences $(x_n)_{n\in\NN}$ and $(y_n)_{n\in\NN}$ are in 
$V$ and $V^{\bot}$, respectively, and the following hold for 
some $\overline{x}\in Z$ and $\overline{y}\in
V^{\bot}\cap(A\overline{x}+P_VB\overline{x})$.
\begin{enumerate}
\item $x_n\weakly \overline{x}$ and $y_n\weakly \overline{y}$.
\item $x_{n+1}-x_n\to 0$ and $y_{n+1}-y_n\to 0$.
\item $P_VBx_n\to P_VB\overline{x}$
\end{enumerate}
\end{corollary}
\begin{proof}
For every $n\in\NN$, set $q_n=(s_n-p_n)/\gamma$. It follows from
\eqref{e:algo} that
\begin{equation}
\begin{cases}
\gamma q_n=s_n-p_n\in\gamma Ap_n\\
s_n=p_n+\gamma q_n,
\end{cases}
\end{equation}
which yield $x_n-\delta_n\gamma P_VBx_n+\gamma y_n=p_n+\gamma q_n$,
$p_n-P_Vp_n=P_{V^{\bot}}p_n=\gamma(y_n-P_{V^{\bot}}q_n)$, and $q_n\in
Ap_n$. Therefore, \eqref{e:algo} is a particular case
of \eqref{e:auxprop} when
$\delta_n\equiv1\in\left]0,2(\beta/\gamma)\right[$ and the results
follow from Theorem~\ref{t:1}.
\end{proof}

\begin{remark}
Note that, when $V=\HH$, \eqref{e:algo} reduces to
\begin{equation}
x_{n+1}=x_n+\lambda_n\big(J_{\gamma A}(x_n-\gamma Bx_n)-x_n\big),
\end{equation}
which is the forward--backward splitting with constant step size
(see \cite{Opti04} and the references therein).
\end{remark}

\begin{remark}
\label{r:simil}
Set $a_n\equiv b_n\equiv0$ in Theorem~\ref{t:0}, set
$\gamma\in\left]0,2\beta\right[$ and
$\delta_n\equiv1$ in Theorem~\ref{t:1}, and let
$(\lambda_n)_{n\in\NN}$ be a sequence in $\left[\varepsilon,1\right]$
for some $\varepsilon\in\left]0,1\right[$. Moreover denote by
$(x_n^1,y_n^1)_{n\in\NN}$ the sequence in $V\times V^{\bot}$ generated
by Theorem~\ref{t:0} and by $(x_n^2,y_n^2)_{n\in\NN}$ the sequence in
$V\times V^{\bot}$ generated by Theorem~\ref{t:1} when
$x_0^1=x_0^2=x_0\in V$ and $y_0^1=y_0^2=y_0\in V^{\bot}$. Then, for
every $n\in\NN$, $x_n^1=x_n^2$ and $y_n^1=y_n^2$. Indeed,
$x_0^1=x_0^2$ and $y_0^1=y_0^2$ by assumption. Proceeding by
mathematical induction, suppose that $x_n^1=x_n^2=x_n$ and
$y_n^1=y_n^2=y_n$. Hence, we deduce from \eqref{e:algo0}, $a_n\equiv
b_n\equiv0$, and \eqref{e:algo} that 
\begin{align}
x_{n+1}^1&=x_n^1+\lambda_n(P_VJ_{\gamma A}(x_n^1-\gamma
P_VBx_n^1+\gamma y_n^1)-x_n^1)\nonumber\\ 
&=x_n^2+\lambda_n(P_VJ_{\gamma A}(x_n^2-\gamma
P_VBx_n^2+\gamma y_n^2)-x_n^2)\nonumber\\
&=x_{n+1}^2.
\end{align}
Moreover, since $P_{V^{\bot}}=\Id-P_V$, we obtain
\begin{align}
y_{n+1}^1&=y_n^1-(\lambda_n/\gamma)P_{V^{\bot}}J_{\gamma
A}(x_n^1-\gamma
P_VBx_n^1+\gamma y_n^1)\nonumber\\ 
&=y_n^2-(\lambda_n/\gamma)P_{V^{\bot}}J_{\gamma A}(x_n^2-\gamma
P_VBx_n^2+\gamma y_n^2)\nonumber\\
&=y_{n+1}^2,
\end{align}
which yields the result. Therefore, both algorithms are the same in
this case. However, even if both methods are very
similar, they can be used differently depending on the nature of each
problem. Indeed, the algorithm proposed in
Theorem~\ref{t:0} allows for explicit errors in the computation of
the operators involved in the general case and the relaxation
parameters $(\lambda_n)_{n\in\NN}$ are allowed to be greater than
those of the method in Theorem~\ref{t:1}. On the other hand, the
method in Theorem~\ref{t:1} allows for a dynamic step size $\delta_n$
in the general case, which is not permitted in the algorithm proposed
in Theorem~\ref{t:0}.
\end{remark}

\section{Applications}
\label{sec:5}
In this section we study two applications of our algorithms. First we
study the problem of finding a zero of the sum of $m$ maximally
monotone operators and a cocoercive operator and, next, we study the
variational case. Connections with other methods in
this framework are also provided. 

\subsection{Inclusion involving the sum of $m$ monotone operators}
Let us consider the following problem.
\begin{problem}
\label{prob:Fad}
Let $(\mathsf{H},|\cdot|)$ be a real Hilbert space, 
for every $i\in\{1,\ldots,m\}$, 
let $\mathsf{A}_i\colon\mathsf{H}\to 2^{\mathsf{H}}$
be a maximally monotone operator, and
let $\mathsf{B}\colon\mathsf{H}\to\mathsf{H}$ be a $\beta$--cocoercive
operator.
The problem is to
\begin{equation}
\text{find}\quad \mathsf{x}\in\mathsf{H}\quad\text{such that}
\quad \mathsf{0}\in\sum_{i=1}^m\mathsf{A}_i\mathsf{x}
+\mathsf{B}\mathsf{x},
\end{equation}
under the assumption that such a solution exists.
\end{problem}

Problem~\ref{prob:Fad} has several applications in image processing, 
principally in the variational setting (see, e.g.,
\cite{Invp08,Fadi12} and the references therein), variational
inequalities \cite{Tsen90,Tsen91}, partial differential
equations \cite{Merc80}, and economics \cite{Jofr07,Penn12}, among
others. In \cite{Fadi12,Bang12} two different methods for solving 
Problem~\ref{prob:Fad} are proposed. In \cite{Bang12} auxiliary
variables are included for solving a more general problem including
linear transformations and additional strongly monotone operators.
This generality does not exploits the intrinsic properties of
Problem~\ref{prob:Fad} and it restricts the choice of
the parameters involved. On the other hand, the method in
\cite{Fadi12} takes into advantage the structure of the problem, but
involves restricting relaxation parameters and errors. We provide an
alternative version to the latter method, which allows for a wider
class of errors and relaxation parameters. The method is obtained as a
consequence of Theorem~\ref{t:0} and the version obtained from
Theorem~\ref{t:1} is also examined.

Let us provide a connection between
Problem~\ref{prob:Fad} and Problem~\ref{prob:1} via product space techniques.
Let $(\omega_i)_{1\leq i\leq m}$ be real numbers in $\zeroun$ such that
$\sum_{i=1}^m\omega_i=1$, let $\HH$ be the real 
Hilbert space obtained by endowing 
the Cartesian product $\mathsf{H}^m$ with the scalar product and
associated norm respectively defined by 
\begin{equation}
\label{e:prodscal}
\scal{\cdot}{\cdot}\colon (x,y)\mapsto
\sum_{i=1}^m\omega_i\pscal{\mathsf{x}_i}{\mathsf{y}_i}\quad\text{and}
\quad
\|\cdot\|\colon
x\mapsto\sqrt{\sum_{i=1}^m\omega_i\mathsf{|}\mathsf{x}_i\mathsf{|}^2},
\end{equation}
where $x=(\mathsf{x}_i)_{1\leq i\leq m}$ is a generic element of $\HH$.
Define 
\begin{equation}
\label{e:defprodc}
\begin{cases}
V=\menge{x=(\mathsf{x}_i)_{1\leq i\leq m}\in\HH}{\mathsf{x}_1=\cdots=\mathsf{x}_m}\\
j\colon\mathsf{H}\to V\subset\HH\colon\mathsf{x}\mapsto(\mathsf{x},\ldots,\mathsf{x})\\
A\colon\HH\to 2^{\HH}\colon x\mapsto \frac{1}{\omega_1}\mathsf{A}_1\mathsf{x}_1\times\cdots\times
\frac{1}{\omega_m}\mathsf{A}_m\mathsf{x}_m\\
B\colon\HH\to\HH\colon x\mapsto(\mathsf{B}\mathsf{x}_1,\ldots,\mathsf{B}\mathsf{x}_m).
\end{cases}
\end{equation}
\begin{proposition}
\label{p:app1}
Let $\mathsf{H}$, $(\mathsf{A}_i)_{1\leq i\leq m}$, and $\mathsf{B}$ be as in 
Problem~\ref{prob:Fad}, and let $V$, $j$, $A$, and $B$ be as in \eqref{e:defprodc}.
Then the following hold.
\begin{enumerate}
\item\label{p:app1i} $V$ is a closed vector subspace of $\HH$, 
$P_V\colon (\mathsf{x}_i)_{1\leq
i\leq m}\mapsto j(\sum_{i=1}^m\omega_i\mathsf{x}_i)$, and 
\begin{equation}
N_V\colon\HH\to 2^{\HH}\colon x\mapsto
\begin{cases}
V^{\bot}=\menge{x=(\mathsf{x}_i)_{1\leq i\leq m}\in\HH}{\sum_{i=1}^m\omega_i\mathsf{x}_i=\mathsf{0}},
\quad&\text{if}\:\:x\in V;\\
\emp,&\text{otherwise}.
\end{cases}
\end{equation}
\item\label{p:app1ii} $j\colon\mathsf{H}\to V$ is a bijective isometry and $j^{-1}\colon(\mathsf{x},\ldots,\mathsf{x})
\mapsto \mathsf{x}$.
\item\label{p:app1iii} $A$ is a maximally monotone operator and, for
every $\gamma\in\RPP$, 
$J_{\gamma A}\colon
(\mathsf{x}_i)_{1\leq
i\leq m}\mapsto(J_{\gamma\mathsf{A}_i/\omega_i}\mathsf{x}_i)$.
\item\label{p:app1iv} $B$ is $\beta$--cocoercive, 
$B(j(\mathsf{x}))=j(\mathsf{B}\mathsf{x})$, and $B(V)\subset V$.
\item\label{p:app1v} For every $\mathsf{x}\in\mathsf{H}$, $\mathsf{x}$ is a solution to Problem~\ref{prob:Fad}
if and only if $j(\mathsf{x})\in\zer(A+B+N_V)$.
\end{enumerate}
\end{proposition}
\begin{proof}
\ref{p:app1i}\&\ref{p:app1ii}: They follow from \eqref{e:normalcone}
and easy computations.
\ref{p:app1iii}: See \cite[Proposition~23.16]{Livre1}.
\ref{p:app1iv}: Let $x=(\mathsf{x}_i)_{1\leq i\leq m}$ and
$y=(\mathsf{y}_i)_{1\leq i\leq m}$ be in $\HH$. Then, it follows from
\eqref{e:defprodc} and the $\beta$--cocoercivity of $\mathsf{B}$ that
\begin{equation}
\scal{Bx-By}{x-y}=\sum_{i=1}^m\omega_i
\pscal{\mathsf{B}\mathsf{x}_i-\mathsf{B}
\mathsf{y}_i}{\mathsf{x}_i-\mathsf{y}_i}\geq\beta\sum_{i=1}^m
\omega_i\mathsf{|}\mathsf{B}\mathsf{x}_i-\mathsf{B}
\mathsf{y}_i\mathsf{|}^2=\beta\|Bx-By\|^2,
\end{equation}
which yields the cocoercivity of $B$. The other results are clear
from the definition.

\ref{p:app1v}: Let $\mathsf{x}\in\mathsf{H}$. We have
\begin{align}
\mathsf{0}\in\sum_{i=1}^m\mathsf{A}_i\mathsf{x}+\mathsf{B}\mathsf{x}\quad
&\Leftrightarrow\quad\bigg(\exi(\mathsf{y}_i)_{1\leq i\leq m}\in
\overset{m}{\underset{i=1}{\cart}}\mathsf{A}_i\mathsf{x}\bigg)
\quad \mathsf{0}=\sum_{i=1}^m\mathsf{y}_i+\mathsf{B}\mathsf{x}\nonumber\\
&\Leftrightarrow\quad\bigg(\exi(\mathsf{y}_i)_{1\leq i\leq m}\in
\overset{m}{\underset{i=1}{\cart}}\mathsf{A}_i\mathsf{x}\bigg)
\quad \mathsf{0}=\sum_{i=1}^m\omega_i(-\mathsf{y}_i/\omega_i-\mathsf{B}\mathsf{x})\nonumber\\
&\Leftrightarrow\quad\bigg(\exi(\mathsf{y}_i)_{1\leq i\leq m}\in
\overset{m}{\underset{i=1}{\cart}}\mathsf{A}_i\mathsf{x}\bigg)
\quad -(\mathsf{y}_1/\omega_1,\ldots,\mathsf{y}_m/\omega_m)-j(\mathsf{B}\mathsf{x})\in
V^{\bot}=N_V(j(\mathsf{x}))\nonumber\\
&\Leftrightarrow\quad
0\in A(j(\mathsf{x}))+B(j(\mathsf{x}))+N_V(j(\mathsf{x}))\nonumber\\
&\Leftrightarrow\quad j(\mathsf{x})\in\zer(A+B+N_V),
\end{align}
which yields the result.
\end{proof}

The following algorithm solves
Problem~\ref{prob:Fad} and is a direct consequence of
Theorem~\ref{t:0}. 
\begin{proposition}
\label{p:4}
Let $\gamma\in\left]0,2\beta\right[$, let
$\alpha=\max\{2/3,2\gamma/(\gamma+2\beta)\}$,
let $(\lambda_n)_{n\in\NN}$ be a sequence in
$\left]0,1/\alpha\right[$, for every
$i\in\{1,\ldots,m\}$, let
$(\mathsf{a}_n)_{n\in\NN}$ and 
$(\mathsf{b}_{i,n})_{n\in\NN}$ be sequences in
$\mathsf{H}$, and suppose that
\begin{equation}
\label{e:conderroaux}
\sum_{n\in\NN}\lambda_n(1-\alpha\lambda_n)=\pinf\quad\text{and}\quad
\max_{1\leq i\leq m}\sum_{n\in\NN}
\lambda_n\Big(\mathsf{|}\mathsf{a}_n|+|\mathsf{b}_{i,n}|\Big)<\pinf.
\end{equation}
Moreover let $(\mathsf{z}_{i,0})_{1\leq i\leq
m}\in\mathsf{H}^m$
and consider the following routine.
\begin{align}
\label{e:algo2}
(\forall n\in\NN)\quad
&\left\lfloor 
\begin{array}{l}
\mathsf{x}_{n}=\sum_{i=1}^m\omega_i\mathsf{z}_{i,n}\\
\text{\rm For }i=1,\ldots,m\\
\left\lfloor
\begin{array}{l}
\mathsf{s}_{i,n}=2\mathsf{x}_n-\mathsf{z}_{i,n}-\gamma
(\mathsf{B}\mathsf{x}_n+\mathsf{a}_n)\\
\mathsf{p}_{i,n}=J_{\gamma \mathsf{A}_i/\omega_i}\mathsf{s}_{i,n}
+\mathsf{b}_{i,n}\\
\mathsf{z}_{i,n+1}=\mathsf{z}_{i,n}+\lambda_n(\mathsf{p}_{i,n}
-\mathsf{x}_n).\\
\end{array}
\right.\\
\end{array}
\right.
\end{align}
Then, the following hold for 
some solution $\overline{\mathsf{x}}$ to Problem~\ref{prob:Fad}.
\begin{enumerate}
\item $\mathsf{x}_n\weakly\overline{\mathsf{x}}$.
\item $\mathsf{B}\mathsf{x}_n\to\mathsf{B}\overline{\mathsf{x}}$.
\item $\mathsf{x}_{n+1}-\mathsf{x}_n\to 0$.
\end{enumerate}
\end{proposition}
\begin{proof}
Set, for every $n\in\NN$, $x_n=j(\mathsf{x}_n)$,
$a_n=j(\mathsf{a}_n)$, $b_n=(\mathsf{b}_{i,n})_{1\leq i\leq m}$,
$y_n=(\mathsf{y}_{i,n})_{1\leq i\leq m}$,
$z_n=(\mathsf{z}_{i,n})_{1\leq i\leq m}$,
$p_n=(\mathsf{p}_{i,n})_{1\leq i\leq m}$, and 
$q_n=(\mathsf{q}_{i,n})_{1\leq i\leq m}$.
It follows from Proposition~\ref{p:app1}\ref{p:app1i} and
\eqref{e:algo2} that, for every $n\in\NN$, $x_n=P_Vz_n$. Hence,
it follows from \eqref{e:defprodc} and Proposition~\ref{p:app1} that
\eqref{e:algo2} can be written equivalently as
\begin{align}
\label{e:algo0aux}
(\forall n\in\NN)\quad
&\left 
\lfloor 
\begin{array}{l}
x_n=P_Vz_n\\
y_n=(x_n-z_n)/\gamma\\
s_n=x_n-\gamma P_V\big(Bx_n+a_n\big)+\gamma y_n\\
p_n=J_{\gamma A}s_n+b_n\\
z_{n+1}=z_n+\lambda_n(p_n-x_n).
\end{array}
\right.
\end{align}
Moreover, it follows from \eqref{e:prodscal} and
\eqref{e:conderroaux} that
\begin{equation}
\sum_{n\in\NN}\lambda_n(\|a_n\|+\|b_n\|)=\sum_{n\in\NN}
\lambda_n\Bigg(|\mathsf{a}_n|+\sqrt{\sum_{i=1}^m\omega_i
|\mathsf{b}_{i,n}|^2}\Bigg)\leq
\sum_{n\in\NN}
\lambda_n\Bigg(|\mathsf{a}_n|+\sum_{i=1}^m
|\mathsf{b}_{i,n}|\Bigg)<\pinf.
\end{equation}
Altogether, Theorem~\ref{t:0} and
Proposition~\ref{p:app1}\ref{p:app1v} yield the results.
\end{proof}

\begin{remark}\
\label{r:fad2}
\begin{enumerate}
\item In the particular case when $(\lambda_n)_{n\in\NN}$ is such
that $0<\varliminf\lambda_n\leq\varlimsup\lambda_n<1/\alpha$ and
the errors are summable, the algorithm \eqref{e:algo2} reduces to the
method in \cite{Fadi12}. Condition \eqref{e:conderroaux} allows for
a larger class of errors and relaxation parameters.

\item Set $\mathsf{a}_n\equiv0$, for every
$i\in\{1,\ldots,m\}$, set $\mathsf{b}_{i,n}\equiv0$, let
$\gamma\in\left]0,2\beta\right[$, and let $(\lambda_n)_{n\in\NN}$ be
a sequence in $\left[\varepsilon,1\right]$ for some
$\varepsilon\in\left]0,1\right[$. Then it follows from
Remark~\ref{r:simil} that the algorithm in Proposition~\ref{p:4}
coincides with the routine: let $\mathsf{x}_0\in\mathsf{H}$, let
$(\mathsf{y}_{i,0})_{1\leq i\leq m}\in\mathsf{H}^m$ such that 
$\sum_{i=1}^m\omega_i\mathsf{y}_{i,0}=0$, and set
\begin{align}
\label{e:algoremark}
(\forall n\in\NN)\quad
&\left\lfloor 
\begin{array}{l}
\text{\rm For }i=1,\ldots,m\\
\left\lfloor
\begin{array}{l}
\mathsf{s}_{i,n}=\mathsf{x}_n-\gamma\mathsf{B}
\mathsf{x}_n+\gamma\mathsf{y}_{i,n}\\
\mathsf{p}_{i,n}=J_{\gamma \mathsf{A}_i/\omega_i}\mathsf{s}_{i,n}\\
\mathsf{y}_{i,n+1}=\mathsf{y}_{i,n}+(\lambda_n/\gamma)(\sum_{i=1}
^m\omega_i\mathsf{p}_{i,n}-\mathsf{p}_{i,n})\\
\end{array}
\right.\\
\mathsf{x}_{n+1}=\mathsf{x}_n+\lambda_n(\sum_{i=1}
^m\omega_i\mathsf{p}_{i,n}-\mathsf{x}_n)
\end{array}
\right.
\end{align}
which is the method proposed in Corollary~\ref{c:0} applied to
Problem~\ref{prob:Fad}. In the particular case
when $\mathsf{B}=0$, $\gamma=1$, and $\lambda_n\equiv1$,
\eqref{e:algoremark} reduces to \cite[Corollary~2.6]{Joca09}.

\item It follows from \eqref{e:aux1} that, in the case 
when $B=0$, the method proposed in 
Proposition~\ref{p:4} follows from the iteration 
\begin{equation}
(\forall n\in\NN)\quad z_{n+1}=z_n+\lambda_n(T_\gamma
z_n+b_n-z_n)
\end{equation}
where $A$ and $V$ are defined in \eqref{e:defprodc}. This method is
very similar to the algorithm proposed in \cite[Theorem~2.5]{Joca09}.
Indeed the only difference is that instead of the operator
$T_{\gamma}=(\Id+R_{\gamma A}R_{N_V})/2$ used in
Proposition~\ref{p:4}, in \cite[Theorem~2.5]{Joca09} is used the
operator $(\Id+R_{N_V}R_{\gamma A})/2$.
\end{enumerate}
\end{remark}

\begin{corollary}
\label{c:1}
Let $\gamma\in\RPP$,
let $(\lambda_n)_{n\in\NN}$ be a sequence in
$\left]0,3/2\right[$, for every
$i\in\{1,\ldots,m\}$, let
$(\mathsf{b}_{i,n})_{n\in\NN}$ be sequences in
$\mathsf{H}$, and suppose that
\begin{equation}
\label{e:conderroaux2}
\sum_{n\in\NN}\lambda_n(3-2\lambda_n)=\pinf\quad\text{and}\quad
\max_{1\leq i\leq m}\sum_{n\in\NN}
\lambda_n|\mathsf{b}_{i,n}|<\pinf.
\end{equation}
Moreover, let
$(\mathsf{z}_{1,0},\mathsf{z}_{2,0})\in\mathsf{H}^2$
and consider the following routine.
\begin{align}
\label{e:algo3}
(\forall n\in\NN)\quad
&\left\lfloor 
\begin{array}{l}
\mathsf{x}_{n}=(\mathsf{z}_{1,n}+\mathsf{z}_{2,n})/2\\
\mathsf{p}_{1,n}=J_{2\gamma
\mathsf{A}_1}(\mathsf{z}_{2,n})+\mathsf{b}_{1,n}\\
\mathsf{p}_{2,n}=J_{2\gamma
\mathsf{A}_2}(\mathsf{z}_{1,n})+\mathsf{b}_{2,n}\\
\mathsf{z}_{1,n+1}=\mathsf{z}_{1,n}+\lambda_n(\mathsf{p}_{1,n}
-\mathsf{x}_{n})\\
\mathsf{z}_{2,n+1}=\mathsf{z}_{2,n}+\lambda_n(\mathsf{p}_{2,n}
-\mathsf{x}_{n}).
\end{array}
\right.
\end{align}
Then, the following hold for 
some solution
$\overline{\mathsf{x}}\in\zer(\mathsf{A}_1+\mathsf{A}_2)$.
\begin{enumerate}
\item $\mathsf{x}_n\weakly\overline{\mathsf{x}}$.
\item $\mathsf{x}_{n+1}-\mathsf{x}_n\to 0$.
\end{enumerate}
\end{corollary}
\begin{proof}
Is a direct consequence of Proposition~\ref{p:4} in the particular case
when $m=2$, $B=0$, $\alpha=2/3$, and $\omega_1=\omega_2=1/2$.
\end{proof}

\begin{remark}\
\begin{enumerate}
 \item
The most popular method for finding a zero of the sum of two
maximally 
monotone operators is the Douglas--Rachford splitting
\cite{Lion79,Svai10},
in which the resolvents of the operators involved are computed
sequentially.
In the case when these resolvents are hard to compute, Corollary~\ref{c:1} 
provides an alternative method which computes in parallel both resolvents.
This method is different to the parallel algorithm proposed in 
\cite[Corollary~3.4]{Siopt3}. 
\item For every
$i\in\{1,\ldots,m\}$, set $\mathsf{b}_{i,n}\equiv0$ and let
$(\lambda_n)_{n\in\NN}$ be a sequence in $\left[\varepsilon,1\right]$
for some $\varepsilon\in\left]0,1\right[$. Then it follows from
Remark~\ref{r:simil} that the algorithm in Corollary~\ref{c:1}
coincides with the routine: let $\mathsf{x}_0\in\mathsf{H}$, let
$\mathsf{v}_0\in\mathsf{H}$, and set
\begin{align}
\label{e:algo3rem}
(\forall n\in\NN)\quad
&\left\lfloor 
\begin{array}{l}
\mathsf{s}_{1,n}=\mathsf{x}_n+\gamma\mathsf{v}_{n}\\
\mathsf{s}_{2,n}=\mathsf{x}_n-\gamma\mathsf{v}_{n}\\
\mathsf{p}_{1,n}=J_{2\gamma \mathsf{A}_1}\mathsf{s}_{1,n}\\
\mathsf{p}_{2,n}=J_{2\gamma \mathsf{A}_2}\mathsf{s}_{2,n}\\
\mathsf{v}_{n+1}=\mathsf{v}_n+(\lambda_n/(2\gamma))(\mathsf{p}_{2,n}
-\mathsf{p}_{1,n})\\
\mathsf{x}_{n+1}=(1-\lambda_n)\mathsf{x}_n
+(\lambda_n/2)(\mathsf{p}_{1,n}+\mathsf{p}_{2,n}),
\end{array}
\right.
\end{align}
which is the method proposed in \eqref{e:algoremark} applied to find
a zero of $\mathsf{A}_1+\mathsf{A}_2$ when $\omega_1=\omega_2=1/2$
and $\mathsf{y}_{1,n}=-\mathsf{y}_{2,n}=\mathsf{v}_n$.
\end{enumerate}

\end{remark}

\subsection{Variational case}
We apply the results of the previous sections to
minimization problems. Let us first recall some standard notation 
and results \cite{Livre1,Zali02}. 
We denote by $\Gamma_0(\HH)$ be the class 
of lower semicontinuous convex functions $f\colon\HH\to\RX$ such that
$\dom f=\menge{x\in\HH}{f(x)<\pinf}\neq\emp$. Let
$f\in\Gamma_0(\HH)$. 
The function $f+\|\cdot-z\|^2/2$
possesses a unique minimizer, which is denoted by $\prox_fz$.
Alternatively, 
\begin{equation}
\label{e:prox2}
\prox_f=(\Id+\partial f)^{-1}=J_{\partial f},
\end{equation}
where $\partial f\colon\HH\to 2^{\HH}\colon x\mapsto
\menge{u\in\HH}{(\forall y\in\HH)\;\:\scal{y-x}{u}+f(x)\leq f(y)}$ 
is the subdifferential of $f$, which is a maximally monotone 
operator. Finally, let $C$ be a convex subset of $\HH$. 
The indicator function of $C$ is denoted by $\iota_C$ and its strong
relative interior 
(the set of points in $x\in C$ such that the cone 
generated by $-x+C$ is a closed vector subspace of $\HH$)
by $\sri C$.
The following facts will also be required.
\begin{proposition}
\label{p:cq}
Let $V$ be a closed vector subspace of $\HH$, let 
$f\in\Gamma_0(\HH)$ be such that $V\cap\dom f\neq\emp$, let
$g\colon\HH\to\RR$ be differentiable and convex. 
Then the following hold. 
\begin{enumerate}
\item
\label{p:cqi}
$\zer(\partial f+\nabla g+N_V)\subset\Argmin(f+g+\iota_V)$.
\item
\label{p:cqii}
Suppose that one of the following is satisfied.
\begin{enumerate}
\item
\label{p:cqiiia}
$\Argmin(f+g+\iota_V)\neq\emp$ and $0\in\sri(\dom f-V)$.
\item
\label{p:cqiiib}
$\Argmin(f+g+\iota_V)
\subset\Argmin f
\cap\Argmin(g+\iota_V)\neq\emp$.
\end{enumerate}
Then $\zer(\partial f+\nabla g+N_V)\neq\emp$.
\end{enumerate}
\end{proposition}
\begin{proof}
\ref{p:cqi}: Since $\dom g=\HH$, it follows from \cite[Corollary~16.38(iii)]{Livre1}
that $\partial(f+g)=\partial f+\nabla g$. Hence, it follows
from $V\cap\dom f\neq\emp$, \cite[Proposition~16.5(ii)]{Livre1}, and
\cite[Theorem~16.2]{Livre1} that
$\zer(\partial f+\nabla g+N_V)=\zer(\partial (f+g)+N_V)
\subset\zer(\partial (f+g+\iota_V))=\Argmin(f+g+\iota_V)$. 

\ref{p:cqiiia}:
Since $\dom g=\HH$ yields $\dom(f+g)=\dom f$,
$\sri(\dom f-V)=\sri(\dom(f+g)-\dom\iota_V)$. Therefore, it 
follows from Fermat's rule (\cite[Theorem~16.2]{Livre1}) and 
\cite[Theorem~16.37(i)]{Livre1} that, for every $x\in\HH$,
\begin{align}
\emp\neq\Argmin(f+g+\iota_V)&=\zer\partial\big(f+g+\iota_V\big)\nonumber\\
&=\zer\big(\partial(f+g)+N_V\big)\nonumber\\
&=\zer\big(\partial f+\nabla g+N_V\big).
\end{align}
 
\ref{p:cqiiib}:
Using \cite[Corollary~16.38(iii)]{Livre1} and~\ref{p:cqi}, from standard 
convex analysis we have
\begin{align}
\Argmin f\cap \Argmin(g+\iota_V)&=
\zer\partial f\cap\zer\partial(g+\iota_V)\nonumber\\
&=\zer\partial f\cap\zer(\nabla g+N_V)\nonumber\\
&\subset\zer(\partial f+\nabla g+N_V)\nonumber\\
&\subset\Argmin(f+g+\iota_V).
\end{align}
Therefore, the hypothesis yields
$\zer(\partial f+\nabla g+N_V)=\Argmin f\cap
\Argmin(g+\iota_V)\neq\emp$.
\end{proof}

The problem under consideration in this section is the following.
\begin{problem}
\label{prob:var}
Let $V$ be a closed vector subspace of $\HH$, let $f\in\Gamma_0(\HH)$,
and let $g\colon\HH\to\RR$ be a differentiable
convex function such that $\nabla g$ is $\beta^{-1}$--Lipschitzian.
The problem is to
\begin{equation}
\minimize{x\in V}{f(x)+g(x)}.
\end{equation}
\end{problem}

Problem~\ref{prob:var} has several applications in partial
differential equations \cite[Section~3]{Merc80}, signal and image
processing
\cite{Aujo05,Aujo06,Byrn02,Cham97,Invp08,Daub04,Eick92,Rudi92,Vese03},
and traffic theory \cite{Sico10,Shef85} among other fields.

In the particular case when $V=\HH$, Problem~\ref{prob:var} 
has been widely studied, the forward-backward splitting can solve
it (see \cite{Sico10,Opti04} and the references therein), and several
applications to multicomponent image processing can be found 
in \cite{Nmtm1} and \cite{Jmiv1}. In the case when $g\equiv0$, 
the partial inverse method in \cite{Spin85} solves
Problem~\ref{prob:var} with some applications
to convex programming. In the general setting,
Problem~\ref{prob:var} can be solved by methods developed in
\cite{Siopt3,Invp08,Fadi12} but without exploiting the structure of
the problem. Indeed, in the algorithms presented in
\cite{Siopt3,Invp08} it is necessary to compute $\prox_g=(\Id+\nabla
g)^{-1}$ and, hence, they do not exploit the fact that $\nabla g$ is
single-valued. In \cite{Fadi12} the method proposed computes
explicitly $\nabla g$, however, it generates auxiliary variables 
for obtaining $P_V$ via product space techniques, which may be
numerically costly in problems with a big number of variables. This
is because this method does not exploit the vector subspace
properties of $V$. The following result provides a method which
exploit the whole structure of the problem and it
follows from Proposition~\ref{c:0} applied to optimization problems.

\begin{proposition}
\label{p:var}
Let $\HH$, $V$, $f$, and $g$ be as in Problem~\ref{prob:var},
let $\gamma\in\left]0,2\beta\right[$, let
$\alpha=\max\{2/3,2\gamma/(\gamma+2\beta)\}$,
let $(\lambda_n)_{n\in\NN}$ be a sequence in
$\left]0,1/\alpha\right[$,
let $(a_n)_{n\in\NN}$ and $(b_n)_{n\in\NN}$ be sequences in $\HH$,
and suppose that
\begin{equation}
\label{e:conderro2}
\sum_{n\in\NN}\lambda_n(1-\alpha\lambda_n)=\pinf\quad\text{and}\quad
\sum_{n\in\NN}\lambda_n(\|a_n\|+\|b_n\|)<\pinf
\end{equation}
and that 
\begin{equation}
\label{e:exivar}
\zer(\partial f+\nabla g+N_V)\neq\emp.
\end{equation}
Moreover let $z_0\in\HH$ and set
\begin{align}
\label{e:algo0var}
(\forall n\in\NN)\quad
&\left 
\lfloor 
\begin{array}{l}
x_n=P_Vz_n\\
y_n=(x_n-z_n)/\gamma\\
s_n=x_n-\gamma P_V\big(\nabla g(x_n)+a_n\big)+\gamma y_n\\
p_n=\prox_{\gamma f}s_n+b_n\\
z_{n+1}=z_n+\lambda_n(p_n-x_n).
\end{array}
\right.
\end{align}
Then, the sequences $(x_n)_{n\in\NN}$ and $(y_n)_{n\in\NN}$ are in 
$V$ and $V^{\bot}$, respectively, and the following hold for 
some solution $\overline{x}$ to Problem~\ref{prob:var} and some
$\overline{y}\in
V^{\bot}\cap\big(\partial
f(\overline{x})+P_V\nabla g(\overline{x})\big)$.
\begin{enumerate}
\item\label{p:vari} $x_n\weakly \overline{x}\,$  and  $\,y_n\weakly
\overline{y}$.
\item\label{p:varii} $x_{n+1}-x_n\to 0\,$  and  $\,y_{n+1}-y_n\to 0$.
\item\label{p:variii}
$\sum_{n\in\NN}\lambda_n\|P_V\big(\nabla g(x_n)-\nabla
g(\overline{x})\big)\|^2<\pinf$.
\end{enumerate}
\end{proposition}
\begin{proof}
It follows from Baillon--Haddad theorem \cite{Bail77} (see also
\cite{Jca10}) that $\nabla g$ is $\beta$--cocoercive and, in addition,
$\partial f$
is maximally monotone. Therefore, the results
follow from Theorem~\ref{t:0}, 
Proposition~\ref{p:cq}\ref{p:cqi}, 
and \eqref{e:prox2} by taking 
$A=\partial f$ and $B=\nabla g$.
\end{proof}

\begin{remark}\
\begin{enumerate}
\item Conditions for assuring condition \eqref{e:exivar} are 
provided in Proposition~\ref{p:cq}\ref{p:cqii}.
\item Set $a_n\equiv0$ and $b_n\equiv0$, let
$\gamma\in\left]0,2\beta\right[$, and let $(\lambda_n)_{n\in\NN}$ be
a sequence in $\left[\varepsilon,1\right]$ for some
$\varepsilon\in\left]0,1\right[$. Then it follows from
Remark~\ref{r:simil} that the algorithm in Proposition~\ref{p:var}
coincides with the routine: let $x_0\in V$, let
$y_0\in V^{\bot}$, and set
\begin{align}
\label{e:algovar2}
(\forall n\in\NN)\quad
&\left 
\lfloor 
\begin{array}{l}
s_n=x_n-\gamma P_V \nabla g(x_n)+\gamma y_n\\
p_n=\prox_{\gamma f}s_n\\
y_{n+1}=y_n+(\lambda_n/\gamma)(P_Vp_n-p_n)\\
x_{n+1}=x_n+\lambda_n(P_Vp_n-x_n),
\end{array}
\right.
\end{align}
which is the method proposed in Corollary~\ref{c:0} applied to
Problem~\ref{prob:var}.

\item
Recently in \cite{Cond12} an algorithm is proposed for solving
simultaneously 
\begin{equation}
\label{e:condat}
\minimize{x\in\HH}{f(x)+g(x)+h(Lx)},
\end{equation}
and its dual, where $\GG$ is a real Hilbert space, $h\in\Gamma_0(\GG)$,
and $L\colon\HH\to\GG$ is linear and bounded. In the particular case
when $\GG=\HH$, $h=\iota_V$, and $L=\Id$, \eqref{e:condat} reduces to
Problem~\ref{prob:var}. In this case, the method is different to 
\eqref{e:algo0var} and, additionally, it needs a more restrictive 
condition on the proximity parameter and the gradient step when 
the constants involved are equal.  

\item Consider the problem involving $N$ convex functions
\begin{equation}
\label{e:minfin}
\minimize{\mathsf{x}\in\mathsf{V}}{\sum_{i=1}^N
\mathsf{f}_i(\mathsf{x})+\mathsf{g}(\mathsf{x})},
\end{equation}
where $\mathsf{H}$ is a real Hilbert space, $\mathsf{V}$ is a closed
vector subspace of $\mathsf{H}$, $(\mathsf{f}_i)_{1\leq i\leq N}$ are
functions in $\Gamma_0(\mathsf{H})$, and $\mathsf{g}$ is convex
differentiable with Lipschitz gradient. Under qualification
conditions, \eqref{e:minfin} can be reduced to Problem~\ref{prob:Fad}
with $m=N+1$, for every $i\in\{1,\ldots,N\}$,
$\mathsf{A}_i=\partial\mathsf{f}_i$, $\mathsf{A}_{N+1}=N_\mathsf{V}$,
and $\mathsf{B}=\nabla\mathsf{g}$. Hence, Proposition~\ref{p:4}
provides an algorithm that solves \eqref{e:minfin}, which generalizes
the method in \cite{Fadi12} in this context by allowing a larger class
of relaxation parameters and errors.
\end{enumerate}
\end{remark}

\end{document}